\documentclass[11pt]{amsart}

\usepackage[a4paper,hmargin=3.5cm,vmargin=4cm]{geometry}
\usepackage{amsfonts,amssymb,amscd,amstext}
\usepackage{graphicx}
\usepackage[dvips]{epsfig}

\usepackage{fancyhdr}
\input xy
\xyoption{all}

\usepackage{times}

\usepackage{enumerate}
\usepackage{titlesec}
\usepackage{mathrsfs}

\titleformat{\section}%[display]
{\filcenter\bfseries\large} {\thesection{.}}{0.2cm}{}%[$\vspace*{-1.0cm}$]
\titleformat{\subsection}[runin]
{\bfseries} {\thesubsection{.}}{0.15cm}{}[.]
\titleformat{\subsubsection}[runin]
{\em}{\thesubsubsection{.}}{0.15cm}{}[.]
\usepackage[up,bf]{caption}
\newtheorem{theorem}{Theorem}[section]
\newtheorem{proposition}[theorem]{Proposition}

\newtheorem{lemma}[theorem]{Lemma}

\theoremstyle{definition}

\numberwithin{equation}{section}
\numberwithin{figure}{section}

\newcommand\Oscr{\mathscr{O}}

\newcommand\B{\mathbb{B}}
\newcommand\C{\mathbb{C}}
\newcommand\D{\overline{\mathbb D}}
\newcommand\CP{\mathbb{CP}}
\renewcommand\D{\mathbb D}

\newcommand\N{\mathbb{N}}

\newcommand\R{\mathbb{R}}

\newcommand\cd{\overline{\mathbb D}}

\renewcommand\d{\mathbb D}

\newcommand\igot{\mathfrak{i}}

\renewcommand\igot{\mathfrak{i}}

\renewcommand\imath{\igot}

\newcommand\hra{\hookrightarrow}

\newcommand\di{\partial}

\newcommand\Aut{\mathrm{Aut}}

\usepackage{color}

\begin{document}

\fancyhead[LO]{Hyperbolic complex contact structures on $\C^{2n+1}$}
\fancyhead[RE]{F.\ Forstneri\v c} 
\fancyhead[RO,LE]{\thepage}

\thispagestyle{empty}

%% Title
\vspace*{1cm}
\begin{center}
{\bf\LARGE Hyperbolic complex contact structures on $\C^{2n+1}$}

\vspace*{0.5cm}

%% Authors
{\large\bf  Franc Forstneri\v c} 
\end{center}

%% Addresses and finantial support
%\footnote[0]{\vspace*{-0.4cm}
%}
%% Abstract, keywords, and MSC

\vspace*{1cm}

\begin{quote}
{\small
\noindent {\bf Abstract}\hspace*{0.1cm}
In this paper we construct complex contact structures on 
$\mathbb{C}^{2n+1}$ for any $n\ge 1$ with the property that every holomorphic Legendrian map $\mathbb{C}\to \mathbb{C}^{2n+1}$ is constant. In particular, these contact structures are not globally contactomorphic to the standard complex 
contact structure on $\mathbb{C}^{2n+1}$.

\vspace*{0.2cm}

\noindent{\bf Keywords}\hspace*{0.1cm} complex contact structures, hyperbolicity, Fatou-Bieberbach domains.

\vspace*{0.1cm}

%\noindent{\bf Mathematics Subject Classification (2010)}\hspace*{0.1cm} 32B15, 32H02, 14H50, 53C42.

\noindent{\bf MSC (2010):}\hspace*{0.1cm} 53D10; 32M17, 32Q45, 37J55}
\end{quote}

%%%%%%%%%%
%%%%%%%%%%
%%%%%%%%%%
%%%%%%%%%%
%%%%%%%%%%
%%%%%%%%%%

\section{Introduction and main results} 
\label{sec:intro}

Let $M$ be a complex manifold of odd dimension $2n+1\ge 3$, where $n\in\N=\{1,2,\ldots\}$. 
A holomorphic vector subbundle $\xi \subset TM$ of complex codimension one in the tangent bundle $TM$ 
is a {\em holomorphic contact structure} on $M$ if every point $p\in M$ admits an open neighborhood 
$U\subset M$ such that $\xi|_U=\ker\alpha$ for a holomorphic $1$-form $\alpha$ on $U$ satisfying 
\[
	\alpha \wedge(d\alpha)^n \neq 0. 
\]
A 1-form $\alpha$ satisfying this nondegeneracy condition is called a {\em holomorphic contact form},
and $(M,\xi)$ is a {\em complex contact manifold}. We shall also write $(M,\alpha)$ when $\xi=\ker \alpha$ holds on all of $M$.  
The model is the complex Euclidean space $(\C^{2n+1}_{x_1,y_1,\ldots,x_n,y_n,z},\xi_0=\ker\alpha_0)$ where
 $\alpha_0$ is the the standard complex contact form
\begin{equation}\label{eq:alpha0}
	\alpha_0 = dz + \sum_{j=1}^n x_j \, dy_j.
\end{equation}
By Darboux's theorem, every holomorphic contact form equals $\alpha_0$  in suitably chosen local 
holomorphic coordinates at any given point (see e.g.\ Geiges \cite[Theorem 2.5.1, p.\ 67]{Geiges2008}
for the smooth case and \cite[Theorem A.2]{AlarconForstnericLopez2016Legendrian} for the holomorphic one). 
This standard case has recently been considered by Alarc\'on, L\'opez and the author in \cite{AlarconForstnericLopez2016Legendrian}. 
They proved in particular that every open Riemann surface $R$ admits a proper holomorphic embedding 
$f\colon R\hra (\C^{2n+1},\alpha_0)$ as a {\em Legendrian curve}, meaning that $f^*\alpha_0=0$ holds on $R$.
In the same paper, the authors asked whether there exists a holomorphic contact form $\alpha$ on $\C^{3}$
which is not globally equivalent to the standard form $\alpha_0$ 
(cf.\ \cite[Problem 1.5, p.\ 4]{AlarconForstnericLopez2016Legendrian}). In this paper we provide
such examples in every dimension.

\begin{theorem}\label{th:main}
For every $n\in \N$ there exists a holomorphic contact form $\alpha$ on $\C^{2n+1}$
such that any holomorphic map $f\colon \C\to\C^{2n+1}$ satisfying $f^*\alpha=0$ is constant.
In particular, the complex contact manifold $(\C^{2n+1},\alpha)$ is not contactomorphic to $(\C^{2n+1},\alpha_0)$.
\end{theorem}

Indeed, a contactomorphism sends Legendrian curves to Legendrian curves, and 
$(\C^{2n+1},\xi_0)$ admits plenty of embedded Legendrian complex lines $\C\hra\C^{2n+1}$.
Indeed, given a point $p=(x_0,y_0,z_0)\in\C^3$ and a  vector
$\nu=(\nu_1,\nu_2,\nu_3) \in \ker\alpha_0|_p$, the quadratic map $f\colon \C\to \C^3$ given by
\[
	f(\zeta)=\left(x_0+\nu_1\zeta,  y_0+\nu_2\zeta, z_0+\nu_3\zeta - \nu_1\nu_2 \zeta^2/2 \right)
\]
is a holomorphic Legendrian embedding satisfying $f(0)=p$ and $f'(0)=\nu$.

We expect that our construction actually gives many nonequivalent holomorphic contact structures on $\C^{2n+1}$; 
however, at this time we do not know how to distinguish them. 
Eliashberg showed that on $\R^3$ there exist countably many isotopy classes of smooth contact structures 
\cite{Eliashberg1989IM,Eliashberg1993IMRN}. His classification is based on the study of 
{\em overtwisted disks} in contact 3-manifolds; it is not clear whether a similar invariant 
could be used in the complex case.

In order to prove Theorem \ref{th:main}, we consider the {\em directed Kobayashi metric} associated to a 
contact complex manifold $(M,\xi)$. Let $\D=\{\zeta\in \C: |\zeta|<1\}$ denote the open unit disk.
Given a  holomorphic subbundle $\xi\subset TM$, we say that a holomorphic disk $f\colon \D\to M$ is 
{\em tangential to $\xi$} or {\em horizontal}  if 
\[
	f'(\zeta) \in \xi|_{f(\zeta)} \quad \text{holds for all}\ \ \zeta\in\D.
\]
Consider the function $\xi\to\R_+$ given for any point  $p\in M$ and vector $v\in \xi_p$ by
\[
 	|v|_{\xi} = \inf \left\{\frac{1}{|\lambda|}: \ \exists f\colon \D\to M\ \text{horizontal},\ f(0)=p,\ f'(0)=\lambda v\right\}.
\]
When $\xi=TM$, this is the Kobayashi length of the tangent vector $v\in T_pM$, and its integrated version 
is the Kobayashi metric on $M$  (cf.\ Kobayashi \cite{Kobayashi1998Book2,Kobayashi2005Book1}). 
The directed version of the Kobayashi metric was studied by Demailly \cite{Demailly1995PSPM} and several other authors,
mainly on complex projective manifolds. More general metrics, obtained by integrating a Riemannian metric 
along horizontal curves in a smooth directed manifold $(M,\xi)$, have been studied by Gromov \cite{Gromov1996PM} 
under the name {\em Carnot-{C}arath\'eodory metrics}. (See also Bella{\"{\i}}che \cite{Bellaiche1996PM}.) 
For this reason, we propose the name {\em Carnot-{C}arath\'eodory-Kobayashi metric}, or {\em CCK metric},  for the 
pseudodistance function $d_\xi\colon M\times M\to \R_+$ defined by 
\begin{equation}\label{eq:dxi}
	d_\xi(p,q)=\inf_\gamma \int_0^1 |\gamma'(t)|_{\xi} \, dt,\quad p,q\in M,
\end{equation}
where the infimum is over all piecewise smooth paths $\gamma\colon [0,1]\to M$ 
satisfying $\gamma(0)=p$, $\gamma(1)=q$ and $\gamma'(t)\in\xi_{\gamma(t)}$ for all $t\in[0,1]$.
(By Chow's theorem \cite{Chow1939MA}, a horizontal path connecting any given pair of points in $M$
exists when the repeated commutators of vector fields tangential to $\xi$ span the tangent space of $M$ at every 
point. A discussion and proof of Chow's theorem can also be found in Gromov's paper 
\cite[p.\ 86 and p.\ 113]{Gromov1996PM}. Another source is Sussman \cite{Sussmann1973TAMS,Sussmann1973BAMS}.)

The directed complex manifold $(M,\xi)$ is said to be  {\em (Kobayashi) hyperbolic} if $d_\xi$ given by \eqref{eq:dxi} 
is a distance function on $M$ (i.e., if $d_\xi(p,q)>0$ holds for all pairs of distinct points $p,q\in M$), and is 
{\em complete hyperbolic} if $d_\xi$ is a complete metric on $M$. Clearly, the directed Kobayashi metric 
on $(M,\xi)$ dominates the standard Kobayashi metric on $M$. 

Now, Theorem \ref{th:main} is an obvious  corollary to the following result.

%
%  Main Theorem 2
%
\begin{theorem}\label{th:main2}
For every $n\in \N$ there exists a holomorphic contact form $\alpha$ on $\C^{2n+1}$ such that the 
complex contact manifold $(\C^{2n+1},\xi=\ker\alpha)$ is Kobayashi hyperbolic.
\end{theorem}

The contact $1$-forms that we shall construct in the proof of Theorem \ref{th:main2} are of the form
\[
	\alpha= \Phi^* \alpha_0
\]
where $\alpha_0$ is the standard contact form \eqref{eq:alpha0} and $\Phi\colon \C^{2n+1}\hra \C^{2n+1}$
is a {\em Fatou-Bieberbach map}, i.e., an injective holomorphic map from $\C^{2n+1}$ onto a proper subdomain
$\Omega=\Phi(\C^{2n+1})\subsetneq \C^{2n+1}$ such that $(\Omega,\alpha_0|_\Omega)$ is a hyperbolic 
contact manifold. Let us describe this construction. Let $C_N>0$ for $N\in\N$ be a sequence diverging to $+\infty$ and 
\begin{equation}\label{eq:K}
	K=\bigcup_{N=1}^\infty 2^{N-1} b\D^{2n}_{(x,y)} \times C_N\overline \D_z.
\end{equation}
Here, $b\D^{2n}_{(x,y)}\subset \C^{2n}$ denotes the boundary of the unit polydisk in the $(x,y)$-space 
and $\overline\D_z$ is the closed unit disk in the $z$ direction. Thus, $K$ is the union of a sequence of compact cylinders
$K_N=2^{N-1} b\D^{2n}_{(x,y)} \times C_N\overline \D_z$ tending to infinity in all directions.
Theorem \ref{th:main2} follows immediately from the following two results of possible independent interest.
In both results, $K$ is the set given by \eqref{eq:K}.

\begin{proposition} \label{prop:hyperbolic}
If $C_N\ge n2^{3N+1}$ holds for all $N\in\N$ then the domain $\Omega_0=\C^{2n+1}\setminus K$ is $\alpha_0$-hyperbolic.
(Here, $\alpha_0$ is the contact form \eqref{eq:alpha0}.)
\end{proposition}

\begin{proposition} \label{prop:avoiding}
For every choice of constants $C_N>0$ there exists a Fatou-Bieberbach domain 
$\Omega\subset \C^{2n+1}\setminus K$.
\end{proposition}

Indeed, if a domain $\Omega_0\subset \C^{2n+1}$ is $\alpha_0$-hyperbolic then so is
any subdomain $\Omega\subset \Omega_0$. Furthermore, 
a biholomorphic map $\Phi\colon \C^{2n+1} \to\Omega$ is an isometry in the directed Kobayashi metric
from the contact manifold $(\C^{2n+1},\alpha)$ with $\alpha=\Phi^*\alpha_0$ onto the contact manifold 
$(\Omega,\alpha_0)$. Since $(\Omega,\alpha_0)$ is hyperbolic by Proposition  \ref{prop:hyperbolic}, 
Theorem \ref{th:main2} follows.

Proposition \ref{prop:hyperbolic} is proved in Section \ref{sec:hyperbolic}; the proof uses 
Cauchy estimates and the explicit expression \eqref{eq:alpha0} for the standard contact form $\alpha_0$.
The set $K$ given by \eqref{eq:K} presents obstacles which impose a limitation on the 
size of holomorphic $\alpha_0$-Legendrian disks. 

Proposition \ref{prop:avoiding} is a special case of Theorem \ref{th:FB} 
which provides a more general result concerning the possibility of avoiding certain unions
of cylinders in $\C^n$ by Fatou-Bieberbach domains.
Its proof is inspired by a result of Globevnik \cite[Theorem 1.1]{Globevnik1997ARKMATH}
who constructed Fatou-Bieberbach domains in $\C^n$ whose intersection with a ball
$R\B^n$ for a given $R>0$ is approximately equal to the intersection of the cylinder $\D^{n-1}\times \C$  
with the same ball. His result implies that one can avoid any cylinder $K_N$
in the set $K$ \eqref{eq:K} by a Fatou-Bieberbach domain $\Omega$. We shall improve the construction so that
$\Omega$ avoids all cylinders $K_N$ at the same time. For this purpose we will use a sequence of holomorphic
automorphisms $\theta_k\in\Aut(\C^n)$ such that the sequence of their compositions
$\Theta_k=\theta_k\circ \cdots\circ \theta_1$ converges on a certain domain $\Omega$ and diverges to infinity
on the set $K$; hence $K\cap\Omega=\emptyset$. 
We ensure in addition that each $\theta_k$ approximates the identity map on the polydisk $k\cd^n$,
and hence the limit $\Theta=\lim_{k\to\infty}\Theta_k \colon \Omega\to\C^{2n+1}$ 
is a biholomorphic map of $\Omega$ onto $\C^{2n+1}$.  

Several interesting questions remain open. One is whether there exists a
{\em complete hyperbolic} complex contact structure on $\C^{2n+1}$.  
Another is  whether there exist {\em algebraic} contact forms $\alpha$
on $\C^{2n+1}$ (i.e., with polynomial coefficients) such that $(\C^{2n+1},\alpha)$ is hyperbolic. 
(Our construction only furnishes transcendental examples.) If so, what is the minimal degree of such examples,
and for which degrees is a generic (or very generic) contact form hyperbolic?
In the integrable case, for affine algebraic and projective manifolds, this
is the famous {\em Kobayashi Conjecture}; see Demailly \cite{Demailly2015arxiv}, 
Brotbek \cite{Brotbek2016ARXIV} and Deng \cite{Deng2016ARXIV} for recent results
on this subject. 

Perhaps the most ambitious question is to classify complex contact structures on Euclidean spaces
up to isotopy, in the spirit of Eliashberg's classification \cite{Eliashberg1989IM,Eliashberg1993IMRN}
of smooth contact structures on $\R^3$.

Holomorphic contact structures on compact complex manifolds
$M=M^{2n+1}$ seem much better understood than those on open manifolds; see for example
the paper by LeBrun \cite{LeBrun1995} and the references therein. 
In particular, the space of all holomorphic contact subbundles of $TM$, if nonempty,
is a connected complex manifold \cite[p.\ 422]{LeBrun1995}. 
Furthermore, if $M$ is simply connected then any two holomorphic contact structures
on $M$ are equivalent via some holomorphic automorphism of $M$ \cite[Proposition 2.3]{LeBrun1995}. In particular,
the only complex contact structure on the projective space $\CP^{2n+1}$ 
(up to projective linear automorphisms) is the standard one, given in homogeneous coordinates
by the $1$-form $\theta=\sum_{j=0}^n (z_jdz_{n+j+1}- z_{n+j+1}dz_j)$.
This structure is obtained by contracting the holomorphic symplectic form 
$\omega=\sum_{j=0}^{n} dz_j\wedge dz_{n+j+1}$ on $\C^{2n+2}$
with the radial vector field $\sum_{k=0}^{2n+1} z_k\frac{\di}{\di z_k}$. 
Its restriction to any affine chart $\C^{2n+1} \subset \CP^{2n+1}$ is 
equivalent to the standard contact structure given by \eqref{eq:alpha0}. It follows 
that the projective space $\CP^{2n+1}$ does not carry any hyperbolic complex contact structures.

%%%%%%%%%%
%%%%%%%%%%
%%%%%%%%%%
%%%%%%%%%%
%%%%%%%%%%
%%%%%%%%%%

\section{Hyperbolic contact structures on domains in $\C^{2n+1}$} 
\label{sec:hyperbolic}

In this section we prove Proposition \ref{prop:hyperbolic}. For simplicity of notation we consider 
the case $n=1$; the same proof applies in every  dimension. 

Thus, let $(x,y,z)$ be complex coordinates on $\C^3$ and $\alpha_0=dz+xdy$ be the standard contact 
form \eqref{eq:alpha0} on $\C^3$. Recall that $\D=\{\zeta\in \C: |\zeta|<1\}$ and $\overline \D=\{\zeta\in \C: |\zeta|\le 1\}$.
The definition of the directed Kobayashi metric shows that Proposition \ref{prop:hyperbolic} is an immediate 
corollary to the following lemma.

%
%  The lemma
%
\begin{lemma}\label{lem:estimate}
Assume that $C_N\ge 2^{3N+1}$ for every $N\in\N$ and let
\[
	K=\bigcup_{N=1}^\infty 2^{N-1} b\D^{2}_{(x,y)} \times C_N\overline \D_z.
\]
For every holomorphic $\alpha_0$-horizontal disk $f(\zeta)=(x(\zeta), y(\zeta),z(\zeta)) \in \C^3\setminus K$
$(\zeta\in\D)$ with $f(0)\in 2^{N_0} \D^3$ for some $N_0\in \N$ we have the estimates
\begin{equation}\label{eq:est}
	|x'(0)| < 2^{N_0+1},\quad |y'(0)| < 2^{N_0+1}, \quad  |z'(0)| < 2^{2N_0+1}.
\end{equation}
\end{lemma}

\begin{proof}
Replacing $f$ by the disk $\zeta\mapsto f(r\zeta)$ for some $r<1$ close to $1$ we may assume that
$f$ is holomorphic on $\overline \D$. Pick a number $N\in\N$ with $N>N_0$ such that $|x(\zeta)| < 2^N$ and $|y(\zeta)|<2^N$ 
for all $\zeta\in\overline \D$. By the Cauchy estimates applied with $\delta=2^{-N}$ we then have
\[
	|y'(\zeta)|<2^{2N}\quad \text{and}\quad  |x(\zeta) y'(\zeta)|< 2^{3N}  \quad \text{for}\ \ |\zeta|\le 1-2^{-N}.
\]
Since $f$ is a horizontal disk, we have $z'(\zeta)=-x(\zeta)y'(\zeta)$ for $\zeta\in\D$ and hence 
\[
	|z(\zeta)| \le |z(0)| +\left| \int_0^\zeta xdy\right| < 2^{N_0}+2^{3N} < 2^{3N+1}\le C_N \quad \text{for}\ \ |\zeta|\le 1-2^{-N}.
\]
From this estimate,  the definition of the set $K$ and the fact that $f(\D)\cap K=\emptyset$ it follows that 
\[
	(x(\zeta),y(\zeta))\notin 2^{N-1}b\D^{2} \quad \text{for}\ \ |\zeta|\le 1-2^{-N}.
\]
Since $2^{N-1}b\D^{2}$ disconnects the bisk $2^N\d^2$ and we have 
$(x(0),y(0))\in 2^{N_0}\D^2\subset 2^{N-1}\D^{2}$, we conclude that
\[
	(x(\zeta),y(\zeta))\in 2^{N-1}\D^{2} \quad \text{for}\ \ |\zeta|\le 1-2^{-N}.
\]
If $N-1>N_0$, we can repeat the same argument with the restricted horizontal disk
$f\colon (1-2^{-N})\cd \to   \C^3$ to obtain
\[
		(x(\zeta),y(\zeta))\in 2^{N-2}\D^{2} \quad \text{for}\ \ |\zeta|\le 1-2^{-N}-2^{-(N-1)}.
\]
After finitely steps of the same kind we get that 
\[
	(x(\zeta),y(\zeta))\in 2^{N_0}\D^{2} \quad \text{for}\ \ |\zeta|\le 1-2^{-N}-\ldots - 2^{-(N_0+1)}.
\]
Since $2^{-N}+\ldots + 2^{-(N_0+1)}<1/2$, we see that 
$(x(\zeta),y(\zeta))\in 2^{N_0}\D^{2}$ for $|\zeta|\le  1/2$. Applying once again the Cauchy estimates
gives $|x'(0)|, |y'(0)|\le 2^{N_0+1}$ and hence $|z'(0)|=|x(0)y'(0)| \le 2^{2N_0+1}$; these are  
precisely the conditions in \eqref{eq:est}.
\end{proof}

%%%%%%%%%%
%%%%%%%%%%
%%%%%%%%%%
%%%%%%%%%%
%%%%%%%%%%
%%%%%%%%%%

\section{Fatou-Bieberbach domains avoiding a union of cylinders}  \label{sec:FB}

In this section we prove the following result on avoiding certain closed cylindrical sets in $\C^n$ by Fatou-Bieberbach domains.
This includes Proposition  \ref{prop:avoiding} as a special case.

\begin{theorem} \label{th:FB}
Let $0<a_1<b_1<a_2<b_2<\cdots$ and $c_i>0$ be sequences of real numbers such that 
$\lim_{i\to \infty}a_i = \lim_{i\to \infty} b_i =+\infty$. Let $n>1$ be an integer and 
\begin{equation}\label{eq:K1}
	K= \bigcup_{i=1}^\infty \left( b_i \cd^{n-1}\setminus a_i \D^{n-1}\right) \times c_i\overline \D \subset \C^n.
\end{equation}
Then there exists a Fatou-Bieberbach domain $\Omega\subset \C^n\setminus K$.
\end{theorem}

As said in the Introduction, the proof is inspired by \cite[proof of Theorem 1.2]{Globevnik1997ARKMATH} to a certain point
and is based on the so called push-out method.
Since the set $K$ \eqref{eq:K1} is noncompact, the construction of automorphisms used in the proof
is somewhat more involved in our case.  On the other hand, since our goal is merely to avoid 
$K$ by a Fatou-Bieberbach domain, and not to approximate a given cylinder as Globevnik
did in \cite{Globevnik1997ARKMATH}, the construction is less precise in certain other aspects.

\begin{proof}
We denote by $\Aut(\C^n)$ the group of all holomorphic automorphisms of $\C^n$.
We first give the proof for $n=2$ and explain in the end how to treat the general case.

Let $(z_1,z_2)$ be complex coordinates on $\C^2$, and let $K=K_1$ be the set \eqref{eq:K1}. 
Up to a dilation of coordinates, we may assume without loss of generality that $a_1>1$. 

Pick sequence $\epsilon_k\in (0,1)$ satisfying $\sum_{k=1}^\infty \epsilon_k < +\infty$,
We shall construct sequences of automorphisms $\phi_k,\psi_k\in\Aut(\C^2)$ $(k\in \N)$ of the following form:
\begin{equation}\label{eq:phipsi}
	\phi_k(z_1,z_2)=(z_1,z_2+f_k(z_1)),\quad \psi_k(z_1,z_2)=(z_1+g_k(z_2),z_2),
\end{equation}
where $f_k$ and $g_k$ are suitably chosen entire functions on $\C$ to be specified. Set 
\begin{equation}\label{eq:theta}
	\theta_k=\psi_k\circ\phi_k,\quad \Theta_k=\theta_k\circ\cdots\circ \theta_1,\quad k\in\N.
\end{equation}
We will also ensure that for every $k\in\N$ we have
\[ 
	|\theta_k(z)-z| < \epsilon_k \quad \text{for} \ \ z\in k\cd^2.
\]
Granted the last condition, it follows (cf.\ \cite[Proposition 4.4.1 and Corollary 4.4.2]{Forstneric2011book})  that the sequence
$\Theta_k\in \Aut(\C^2)$ converges uniformly on compacts in the open set 
\[
	\Omega= \bigcup_{k=1}^\infty \Theta_k^{-1}(k\d^2) = \{z\in\C^2: (\Theta_k(z))_{k\in \N}\ \ \text{is a bounded sequence}\}
\] 
to a biholomorphic map $\Theta=\lim_{k\to\infty}\Theta_k \colon \Omega\to\C^2$ of $\Omega$ onto $\C^2$. 
We will also ensure that 
\begin{equation}\label{eq:Thetak}
	|\Theta_k(z)| \to + \infty \quad \text{for all points}\ \ z\in K,
\end{equation}
and hence $K\cap \Omega=\emptyset$. This will prove the theorem when $n=2$.

We begin by explaining how to choose the first two maps $\phi_1$ and $\psi_1$; all subsequent
steps will be analogous. Set $b_0=1$. Pick a sequence $r_j$ satisfying $b_{j-1}<r_j<a_{j}$ for all $j=1,2,\ldots$. 
Let $N_j\in\N$ be a sequence of integers to be specified later. Set 
\[
	f(\zeta) = \sum_{j=1}^\infty \left(\frac{\zeta}{r_j}\right)^{N_j}.
\]
This function will define the first automorphism $\phi_1$ (cf.\ \eqref{eq:phipsi}).
Let $f_i(\zeta) = \sum_{j=1}^i \left(\frac{\zeta}{r_j}\right)^{N_j}$ denote the $i$-th partial sum
of the series defining $f(\zeta)$, where we set $f_0=0$. 
By choosing the exponent $N_i$ big enough, we can ensure that the summand
$(\zeta/r_i)^{N_i}$ is arbitrarily small on the disk $b_{i-1}\cd$ and is arbitrarily big on the annulus 
\begin{equation}\label{eq:annulus}
	A_i:=b_i \cd\setminus a_i \D=\{\zeta: a_i \le |\zeta| \le b_i\}. 
\end{equation}
In particular, we may ensure that for every $i\in\N$ we have
\begin{equation}\label{eq:est1}
	\sup_{|\zeta| \le b_{i-1}} \left| \frac{\zeta}{r_{i}}\right|^{Ni} <2^{-i-1} \epsilon_1.
\end{equation}
It follows that the power series defining $f(\zeta)$ converges on all of $\C$ and satisfies
\begin{equation}\label{eq:est2}
	\sup_{|\zeta| \le b_{i-1}}|f(\zeta) - f_{i-1}(\zeta)| <2^{-i} \epsilon_1,\quad i\in \N.
\end{equation}
Note that the inequalities \eqref{eq:est1} and \eqref{eq:est2} persist if we increase the exponents $N_i$.
We can inductively choose the sequence $N_i\in \N$ to grow fast enough such that the following
inequalities hold for every $i\in\N$ with an increasing sequence of numbers $M_i \ge i+1$: 
\begin{equation}\label{eq:est3}
	\sup_{|\zeta| \le b_{i-1}}|f_{i-1}(\zeta)| + c_{i-1} + \epsilon_1 < M_i <
	\inf_{\zeta\in A_i} \left( \left| \frac{\zeta}{r_i} \right|^{N_i} - |f_{i-1}(\zeta)| \right) - c_i - \epsilon_1.
\end{equation}
(Recall that $A_i$ is the annulus \eqref{eq:annulus}.
Here, $c_0\ge 0$ is arbitrary while $c_i>0$ for $i\in\N$ are the constants in the definition \eqref{eq:K1} of the set $K$.)
In view of the inequalities \eqref{eq:est1}, \eqref{eq:est2} and \eqref{eq:est3} 
there exist numbers $\beta_{i-1} < \alpha_i$ such that for all $i\in\N$ we have
\begin{equation}\label{eq:est4}
	\sup_{|\zeta| \le b_{i-1}}|f(\zeta)| + c_{i-1}  <  \beta_{i-1} <  M_i < \alpha_i < \inf_{\zeta\in A_i} |f(\zeta)|  - c_i.
\end{equation}
This gives increasing sequences $0<\beta_0< \alpha_1<\beta_1<\alpha_2<\beta_2<\cdots$ diverging to $\infty$. Set 
\[
	\phi_1(z_1,z_2)=(z_1,z_2+f(z_1)).
\]
The right hand side of  \eqref{eq:est4} shows that for every point $z=(z_1,z_2)\in A_i\times c_i\cd$ we have
\[
	|z_2+f(z_1)| \ge |f(z_1)|- c_i >\alpha_i,
\]
while the left hand side of \eqref{eq:est4} gives
\[
	|z_2+f(z_1)| \le c_i + |f(z_1)| < \beta_i. 
\]
Since these inequalities hold for every $i\in\N$, it follows that 
\[
	\phi_1(K) \subset L:=  \bigcup_{i=1}^\infty b_i\overline \D \times 
	\left( \beta_i \cd\setminus \alpha_i \d\right) \subset \C^2.
\]
Note that the set $L$ is of the same kind as $K$ \eqref{eq:K1} with the reversed roles 
of the variables, i.e., the cylinders in $L$ are horizontal instead of vertical.
Furthermore, since the sequence $\alpha_i$ is increasing and $\alpha_1 > M_1\ge 2$ 
by \eqref{eq:est4}, we also see that 
\[
	L\cap (\C\times 2\cd) =\emptyset.
\]
The same argument as above with the set $L$ furnishes a shear automorphism
\[
	\psi_1(z_1,z_2)=(z_1+g(z_2),z_2)
\]
for some $g\in\Oscr(\C)$ (cf.\ \eqref{eq:phipsi}) and a set $K_2$ of the same kind as $K=K_1$ \eqref{eq:K1} 
(this time again with vertical cylinders) such that, setting $\theta_1:=\psi_1\circ\phi_1\in\Aut(\C^2)$, we have 
\begin{equation}\label{eq:step1}
	\theta_1(K_1) \subset K_2, \quad K_2 \cap 2\cd^2 =\emptyset,
	\quad \sup_{z\in \cd^2} |\theta_1(z)-z| <\epsilon_1.
\end{equation}

Continuing inductively,  we find a sequence of automorphisms $\theta_k\in\Aut(\C^2)$ and of closed
sets $K_k\subset \C^2$ of the form \eqref{eq:K1} such that for every $k\in\N$ we have 
\begin{equation}\label{eq:condition-k}
	\theta_k(K_k) \subset K_{k+1},\quad K_k \cap k\cd^2 =\emptyset, \quad
	\sup_{z\in k\cd^2} |\theta_k(z)-z| <\epsilon_k.
\end{equation}
Each step of the recursion is of exactly the same kind as the initial one. This implies that 
\[
	\Theta_k(K)\subset K_{k+1}\subset \C^2\setminus (k+1)\cd^2,\quad k\in\N
\] 
and hence  \eqref{eq:Thetak} also holds. This completes the proof when $n=2$. 

Suppose now that $n>2$.  In this case, each automorphism $\theta_k=\psi_k\circ\phi_k\in\Aut(\C^n)$
in the sequence \eqref{eq:theta} is a composition of two shear-like maps of the form
\begin{eqnarray*}
	\phi_k(z_1,z_2,\ldots,z_n) &=& \bigl(z_1,z_2+f_k(z_1),z_3+f_k(z_2),\ldots, z_n+f_k(z_{n-1})\bigr), \\
	\psi_k(z_1,z_2,\ldots,z_n) &=& \bigl(z_1+g_k(z_2),z_2+g_k(z_3),\ldots, z_{n-1}+g_k(z_n), z_n)\bigr).
\end{eqnarray*}
A suitable choice of entire functions $f_k,g_k \in\Oscr(\C)$ ensures as before that condition
\eqref{eq:condition-k} holds for each $k$ (with $\cd^2$ replaced by $\cd^n$). We leave the details to an interested reader.
Further details in the case $n>2$ are also available in \cite[proof of Theorem 1.2]{Globevnik1997ARKMATH}.
\end{proof}

%%%%%%%%%%
%%%%%%%%%%
%%%%%%%%%%
%%%%%%%%%%   THANKS
%%%%%%%%%%
%%%%%%%%%%

\subsection*{Acknowledgements}
Research on this work is partially  supported  by the research program P1-0291 and the grant J1-7256 from 
ARRS, Republic of Slovenia.  

I wish to thank Yakov Eliashberg for having provided some of the references and for 
stimulating discussions, and  Josip Globevnik for having brought to my attention the paper
\cite[Theorem 1.2]{Globevnik1997ARKMATH} whose main idea is employed in the proof of Theorem \ref{th:FB}.

%%%%%%%%%%
%%%%%%%%%%
%%%%%%%%%%
%%%%%%%%%%   THE BIBLIOGRAPHY
%%%%%%%%%%
%%%%%%%%%%

{\bibliographystyle{abbrv} \bibliography{bibADFL}}

%%%%%%%%%%
%%%%%%%%%%
%%%%%%%%%%
%%%%%%%%%%   AFFILIATIONS
%%%%%%%%%%
%%%%%%%%%%

\vspace*{0.5cm}
\noindent Franc Forstneri\v c

\noindent Faculty of Mathematics and Physics, University of Ljubljana, Jadranska 19, SI--1000 Ljubljana, Slovenia

\noindent Institute of Mathematics, Physics and Mechanics, Jadranska 19, SI--1000 Ljubljana, Slovenia

\noindent e-mail: {\tt franc.forstneric@fmf.uni-lj.si}

\end{document}